\documentclass[a4paper,12pt]{amsart}
\usepackage{hyperref,fancyhdr,mathrsfs,amsmath,amscd,amsthm,amsfonts,latexsym,amssymb,stmaryrd}
\usepackage[all]{xy}

\DeclareMathOperator{\dif}{d}

\DeclareMathOperator{\Lie}{\mathcal{L}}

\renewcommand{\H}{\mathscr{H}}
\newcommand{\V}{\mathscr{V}}

\newcommand{\J}{\mathcal{J}}

\renewcommand{\L}{\mathscr{L}}

\def \a{\alpha}

\def \b{\beta}

\def \o{\omega}

\def \phi{\varphi}
\def \Phi{\varPhi}
\def \p{\pi}
\def \r{\rho}

\def \t{\tau}

\def \R{\mathbb{R}}
\def \Hq{\mathbb{H}\,}
\def \C{\mathbb{C}\,}

\def\widecheckg{g^{\hspace*{-2.5pt}\vbox to 5pt{\hbox to
0pt{\LARGE$\check{}$}}}\hspace*{2pt}}

\def\widecheckl{\lambda^{\hspace*{-3.5pt}\vbox to 8pt{\hbox to
0pt{\LARGE$\check{}$}}}\hspace*{2pt}}

\begin{document}

\title{Generalized Quaternionic Manifolds}
\author{Radu Pantilie} 
\email{\href{mailto:radu.pantilie@imar.ro}{radu.pantilie@imar.ro}}
\address{R.~Pantilie, Institutul de Matematic\u a ``Simion~Stoilow'' al Academiei Rom\^ane,
C.P. 1-764, 014700, Bucure\c sti, Rom\^ania}

\newtheorem{thm}{Theorem}[section]
\newtheorem{lem}[thm]{Lemma}
\newtheorem{cor}[thm]{Corollary}
\newtheorem{prop}[thm]{Proposition}

\theoremstyle{definition}

\newtheorem{defn}[thm]{Definition}
\newtheorem{rem}[thm]{Remark}
\newtheorem{exm}[thm]{Example}

\numberwithin{equation}{section}

\maketitle
\thispagestyle{empty}
\vspace{-10mm}
\section*{Abstract}
\begin{quote}
{\footnotesize 
We initiate the study of the generalized quaternionic manifolds by classifying the generalized 
quaternionic vector spaces, and by giving two classes of nonclassical examples of such manifolds. 
Thus, we show that any complex symplectic manifold is endowed with a natural (nonclassical) generalized quaternionic structure, 
and the same applies to the heaven space of any three-dimensional 
Einstein--Weyl space. In particular, on the product $Z$ of any complex symplectic manifold $M$ 
and the sphere there exists a natural generalized complex structure, with respect to which $Z$ 
is the twistor space of~$M$. 
}
\end{quote}

\section*{Introduction} 

\indent 
There are several natural notions which are similar or generalize the complex manifolds. 
One of these is based on the idea that, at each point, instead of a single linear complex structure, 
one considers a quaternionic family of linear complex structures. Then the corresponding integrability condition 
is in the spirit of Twistor Theory, thus obtaining the basic objects of Quaternionic Geometry (see \cite{IMOP} and the references therein).\\ 
\indent  
On the other hand, one may consider a linear complex structure on the direct sum of the tangent and cotangent bundles, 
which is orthogonal with respect to the canonical inner product (corresponding to the natural identification of the tangent bundle with its bidual). 
Then the corresponding integrability condition is provided by a generalization of the usual bracket on vector fields \cite{Courant}\,, thus   
leading to the Generalized Complex Geometry (see \cite{Gua-thesis}\,).\\ 
\indent  
In this note we initiate the natural unification of these two Geometries, under the framework of Twistor Theory.\\ 
\indent 
In Sections \ref{section:gcvs} and \ref{section:brqvs} we review the generalized complex and the quaternionic vector spaces, respectively. 
Then in Section \ref{section:gqvs} we classify the generalized quaternionic vector spaces (Theorem \ref{thm:linear_gq}\,). We mention that, 
although the formulation of this result is elementary, its proof requires a covariant functor which, to any pair formed of a quaternionic 
vector space and a real subspace, associates a coherent analytic sheaf over the Riemann sphere.\\ 
\indent 
Finally, in Section \ref{section:gqm} we introduce the notion of generalized quaternionic manifold and we present two classes of nonclassical examples,  
provided by the complex symplectic manifolds (Example \ref{exm:complex_symplectic}\,) 
and by the heaven space of any three-dimensional Einstein--Weyl space (Example \ref{exm:qKv}\,).

\section{Generalized complex vector spaces} \label{section:gcvs} 

\indent
A \emph{generalized linear complex structure} \cite{Gua-thesis} on a (finite dimensional real) vector space $V$ is a linear
complex structure on $V\times V^*$ which is orthogonal with respect to the inner product corresponding to the canonical
pairing of\/ $V$ and $V^*$.\\
\indent
Before taking a closer look to the generalized linear complex structures, we recall (see \cite{Gua-thesis}\,) the \emph{linear $B$-field transformations}.
These can be defined as the orthogonal transformations of $V\times V^*$ which restrict to the identity on $V^*$. More concretely,
any linear $B$-field transformation is given by $(X,\a)\mapsto(X,\iota_Xb+\a)$\,, for any $(X,\a)\in V\times V^*$, where $b$ is a
two-form on $V$.\\
\indent
Indeed, any linear $B$-field transformation maps $V$ onto an isotropic complement of $V^*$, and any such subspace of $V\times V^*$
is the graph of a unique two-form on $V$ (seen as a linear map from $V$ to $V^*$).\\
\indent
In what follows, we shall use several times this canonical correspondence between linear $B$-field transformations
and isotropic complements of $V^*$ in $V\times V^*$.\\
\indent
Let $\J$ be a generalized linear complex structure on $V$. There are two extreme cases: $\J V^*=V^*$ and $V^*\cap\J V^*=\{0\}$\,.\\
\indent
In the first case, we obtain, inductively, that there exists a $\J$-invariant isotropic complement $W$ of $V^*$ in $V\times V^*$.
Consequently, $\J V^*=V^*$ if and only if, up to a linear $B$-field transformation, we have
\begin{equation*}
\J=
\begin{pmatrix}
J & 0 \\
0 & J^*
\end{pmatrix}
\;,
\end{equation*}
where $J$ is a linear complex structure on $V$ and we have denoted by $J^*$ the opposite of the transpose of $J$.\\
\indent
Similarly, on taking $W=\J V^*$ we obtain that $V^*\cap\J V^*=\{0\}$ if and only if, up to a linear $B$-field transformation, we have
\begin{equation*}
\quad\J=
\begin{pmatrix}
0 & \o^{-1} \\
-\o & 0
\end{pmatrix} \;,
\end{equation*}
where $\o$ is a linear symplectic structure on $V$ (seen as a linear isomorphism from $V$ onto $V^*$).\\ 
\indent
Note that if $\J$ and $\mathcal{K}$ are linear generalized complex structures on $V$ and $W$, respectively,
then the product $\J\times\mathcal{K}$ is defined, in the obvious way, through the isomorphism
$(V\times W)\times(V\times W)^*=(V\times V^*)\times(W\times W^*)$\,.

\begin{prop}[\cite{Gua-thesis}]
Any linear generalized complex structure is, up to a linear $B$-field transformation, the product of the linear generalized
complex structures given by a (classical) linear complex structure and a linear symplectic structure.
\end{prop}
\begin{proof}
If $\J$ is a generalized linear complex structure on $V$ then
$$\bigl(V^*\cap\J V^*\bigr)^{\perp}=\bigl(V^*\bigr)^{\perp}+\J\bigl(V^*\bigr)^{\perp}=V^*+\J V^*\;.$$
\indent
Thus, if\/ $U$ is a complement of\/ $V^*\cap\J V^*$ in $V^*$ then $U\cap\J U=\{0\}$ and $U+\J U$ is nondegenerate
(with respect to the canonical inner product).\\
\indent
Then the orthogonal complement of $U+\J U$ is a nondegenerate complex vector subspace of
$(V\times V^*,\J)$ for which $V^*\cap\J V^*$ is a maximal isotropic subspace.\\
\indent
It follows that, up to a linear $B$-field transformation, $(V,\J)$ is the product of two generalized
complex vector spaces $(V_{\!1},\J_1)$ and $(V_{\!2},\J_2)$ satisfying $\J_1 V_{\!1}^*=V_{\!1}^*$ and $V_{\!2}^*\cap\J_2 V_{\!2}^*=\{0\}$\,.
\end{proof}

\indent 
The following fact will be used later on. 

\begin{lem} \label{lem:fyt} 
Let $\r:E\to U$ be a surjective linear map and let $V=U\times({\rm ker}\r)^*$.\\ 
\indent  
Then any section of $\r$ induces an isomorphism $V\times V^*=E\times E^*$ which 
preserves the canonical inner products. Moreover, any two such isomorphisms differ by a linear $B$-field transformation. 
\end{lem} 
\begin{proof} 
This is a quick consequence of the fact that any section of $\r$ corresponds to a splitting $E=U\times({\rm ker}\r)$\,. 
\end{proof}

\section{Brief review of the quaternionic vector spaces} \label{section:brqvs} 

\indent
Let $\Hq$ be the (unital) associative algebra of quaternions. Its automorphism group is ${\rm SO}(3)$ acting trivially on $\R$
and canonically on ${\rm Im}\Hq$\,.\\
\indent
A \emph{linear hypercomplex structure} on a vector space $E$ is a morphism of associative algebras from $\Hq$ to ${\rm End}(E)$\,.
The automorphism group of $\Hq$ acts on the space of linear hypercomplex structures on $E$ in an obvious way, thus giving the
canonical equivalence relation on it.\\ 
\indent
A \emph{linear quaternionic structure} on a vector space is an equivalence class of linear hypercomplex structures.
A vector space endowed with a linear quaternionic (hypercomplex) structure is a \emph{quaternionic (hypercomplex) vector space}.\\
\indent
Let $E$ be a quaternionic vector space and let $\r:\Hq\to{\rm End}(E)$ be a representative of its linear quaternionic structure.
Then $Z=\r(S^2)$ is the space of \emph{admissible} linear complex structures on $E$. Obviously, $Z$ depends only of the linear
quaternionic structure of $E$.\\
\indent
Let $E$ and $E'$ be quaternionic vector spaces and let $Z$ and $Z'$ be the corresponding spaces of admissible linear complex structures,
respectively. A linear map $t:E\to E'$ is \emph{quaternionic}, with respect to some function $T:Z\to Z'$, if $t\circ J=T(J)\circ t$\,,
for any $J\in Z$. It follows that if $t\neq0$ then $T$ is unique and an orientation preserving isometry \cite{IMOP}\,.\\
\indent
The basic example of a quaternionic vector space is $\Hq^{\!k}$, $(k\in\mathbb{N})$\,, endowed with the linear quaternionic structure
given by its (left) $\Hq$-module structure. Furthermore, if $E$ is a quaternionic space, $\dim E=4k$\,, then there exists a linear quaternionic isomorphism from $\Hq^{\!k}$ onto $E$.\\
\indent
The classification (see \cite{vq}\,) of the real vector subspaces of a quaternionic vector space is much harder. It involves two
important particular subclasses, dual to each other.

\begin{defn}[\,\cite{fq}\,]
A \emph{linear CR quaternionic structure} on a vector space $U$ is a pair $(E,\iota)$\,, where $E$ is a quaternionic vector space 
and $\iota:U\to E$ is an injective linear map such that ${\rm im}\,\iota+J({\rm im}\,\iota)=E$, for any admissible linear complex structure $J$ on $E$.\\
\indent
A vector space endowed with a linear CR quaternionic structure is a \emph{CR quaternionic vector space}.
\end{defn}

\indent
By duality, we obtain the notion of \emph{co-CR quaternionic vector space}.\\
\indent 
It is convenient to work with the category whose objects are pairs formed of a quaternionic vector space $E$ and a real vector subspace $U$. Then a  
CR quaternionic vector space $(U,E,\iota)$ corresponds to the pair $({\rm im}\,\iota,E)$ (and, by duality, a co-CR quaternionic vector space 
$(U,E,\r)$ corresponds to the pair $({\rm ker}\,\r,E)$\,).\\ 
\indent 
There exists a covariant functor from this category to the category of coherent analytic sheaves over the sphere.  To describe it, 
let $Z\,(=S^2)$ be the space of admissible linear complex structures on $E$ and denote by $E^{0,1}$ the holomorphic vector subbundle
of $Z\times E^{\C}$ whose fibre, over each $J\in Z$, is the eigenspace of $J$ corresponding to $-{\rm i}$\,.\\
\indent
Let $\t$ be the restriction to $E^{0,1}$ of the morphism of trivial holomorphic vector bundles determined by the
projection $E\to E/U$. Denote by $\mathcal{U}_-$ and $\mathcal{U}_+$ the kernel and cokernel of $\t$, respectively.
Then $\mathcal{U}=\mathcal{U}_+\oplus\mathcal{U}_-$ is \emph{the (coherent analytic) sheaf of $(U,E)$} \cite{vq}\,.\\
\indent
For example, $(U,E)$ is a CR quaternionic vector space if and only if $\mathcal{U}=\mathcal{U}_-$ (which is a holomorphic
vector bundle whose Birkhoff--Grothendieck decomposition contains only terms of Chern number at most $-1$). The simplest
concrete examples are $(\Hq,\Hq)$ and $({\rm Im}\Hq,\Hq)$ which correspond to $2\mathcal{O}(-1)$ and $\mathcal{O}(-2)$\,, respectively;
their \emph{duals} are $(0,\Hq)$ and $(\R,\Hq)$ which correspond to $2\mathcal{O}(1)$ and $\mathcal{O}(2)$\,, respectively.\\
\indent
See \cite{fq}\,,\,\cite{fq-2} and \cite{vq} for further examples and the related Twistor Theory. 

\begin{rem} \label{rem:product_cr_q decomp}
1) If $E$ and $E'$ are quaternionic vector spaces endowed with the real vector subspaces
$U$ and $U'$, respectively, such that either the sheaf of $(U,E)$ or the sheaf of $(U',E')$ is torsion free then \cite[Corollary 4.2(i)]{vq}
the product $(U\times U',E\times E')$ is well-defined, up to a linear quaternionic isomorphism (that is, it does not depend of the particular
linear hypercomplex structures, representing the linear quaternionic structures of $E$ and $E'$, used to define $E\times E'$\,).\\
\indent
2) Let $(U,E)$ be a pair formed of a quaternionic vector space $E$ and a real vector subspace $U$; denote by $\mathcal{U}$ the sheaf of $(U,E)$\,.\\
\indent
Then $\mathcal{U}_-$ is the holomorphic vector bundle of a canonically defined CR quaternionic vector space $(U_-,E_-)\subseteq(U,E)$\,.\\
\indent
Also, there exists pairs $(V,F)$ and $(U_t,E_t)$ such that $(V,F)$ corresponds to a co-CR quaternionic vector space
(that is, the projection $E\to E/V$ defines a linear co-CR quaternionic structure on $E/V$), the sheaf\/ $\mathcal{U}_t$ of $(U_t,E_t)$
is the torsion subsheaf of $\mathcal{U}$, and $\mathcal{U}_+=\mathcal{U}_t\oplus\mathcal{V}$\,, $(U,E)=(U_-,E_-)\times(U_t,E_t)\times(V,F)$\,,
where $\mathcal{V}$ is the holomorphic vector bundle of $(V,F)$\,.\\
\indent
Moreover, the filtration $\{0\}\subseteq(U_-,E_-)\subseteq(U_-,E_-)\times(U_t,E_t)\subseteq(U,E)$ is canonical \cite[Corollary 4.2(ii)]{vq} .
\end{rem}

\indent 
We end this section with the following fairly obvious fact which will be used later on. 

\begin{prop} \label{prop:from torsion_to_co-cr_q} 
Let $(U,E)$ and $(V,F)$ be pairs formed of a quaternionic vector space and a real subspace such that the sheaf of $(U,E)$ is torsion,  
whilst $(V,F)$ corresponds to a co-CR quaternionc vector space.\\ 
\indent 
Then any morphism from $(U,E)$ to $(V,F)$ is zero.  
\end{prop}

\section{Generalized quaternionic vector spaces} \label{section:gqvs} 

\indent
The notions of linear quaternionic structure and generalized linear complex structure suggest the following.

\begin{defn}
A \emph{generalized linear quaternionic structure} on a vector space $V$ is a linear quaternionic structure on $V\times V^*$
whose admissible linear complex structures are orthogonal with respect to the inner product.\\
\indent
A \emph{generalized quaternionic vector space} is a vector space endowed with a generalized linear quaternionic structure.
\end{defn}

\indent
There are two basic classes of generalized quaternionic vector spaces.

\begin{exm} \label{exm:cr_q} 
To any co-CR quaternionic vector space $(U,E,\r)$ we associate on $V=U\times({\rm ker}\r)^*$, by using Lemma \ref{lem:fyt}\,,  
a generalized linear quarternionic structure which is unique, up to linear $B$-field transformations.\\ 
\indent  
We, thus, obtain \emph{the generalized quaternionic vector space given by the co-CR quaternionic vector space $(U,E,\r)$}.\\
\indent
Note that, this construction gives a (classical) quaternionic vector space if and only if $\r$ is an isomorphism; equivalently, $V=E$
as quaternionic vector spaces.\\ 
\indent 
By duality, we obtain the generalized quaternionic vector space given by a CR quaternionic vector space.  
\end{exm}

\begin{exm} \label{exm:gc}
Let $(V,J,\o)$ be a vector space endowed with a linear complex structure $J$ and a linear symplectic structure $\o$.
Denote by $\J_J$ and $\J_{\o}$ the generalized linear complex structures on $V$ given by $J$ and $\o$, respectively.\\
\indent
Then $\J_J\,\J_{\o}=-\J_{\o}\,\J_J$ if and only if $\o^{(1,1)}=0$\,. Thus, if $(V,J,\o)$ is a complex symplectic vector space
then $\bigl\{a\J_J+b\J_{\o}+c\J_{\o J}\,|\,(a,b,c)\in S^2\bigr\}$ defines a generalized linear quaternionic structure on $V$.\\
\indent
We, thus, obtain \emph{the generalized quaternionic vector space given by the complex symplectic vector space $(V,J,\o)$}.\\
\indent
Note that, if $a\neq\pm1$ then $a\J_J+b\J_{\o}+c\J_{\o J}$ is given, up to a linear $B$-field transformation,
by a linear symplectic structure.
\end{exm}

\indent
Here is the main result of this section.

\begin{thm} \label{thm:linear_gq} 
Any generalized linear quaternionic vector space is, up to a linear $B$-field transformation, a product of the
generalized quaternionic vector spaces given by a CR quaternionic vector space and a finite family of complex symplectic vector spaces; 
moreover, the factors are unique, up to ordering. 
\end{thm}

\indent
To prove Theorem \ref{thm:linear_gq} we, firstly, consider the case when the sheaf of the pair $(V^*,V\times V^*)$ is torsion free.

\begin{prop} \label{prop:linear_gq_bundle}
Let $V$ be a generalized quaternionic vector space such that the sheaf of $(V^*,V\times V^*)$ is torsion free.\\
\indent
Then, up to a linear $B$-field transformation, $V$ is given by a unique (up to isomorphisms) CR quaternionic vector space.
\end{prop}
\begin{proof}
Let $\mathcal{V}$ be the sheaf of $(V^*,V\times V^*)$ and let $\mathcal{V}_{\pm}$ be its positive/negative parts.\\ 
\indent 
The canonical inner product induces an isomorphism $\mathcal{V}=\mathcal{V}^*$ which is equal to its transpose and preserves the decompositions into the positive 
and negative parts. As the positive/negative parts of $\mathcal{V}^*$ are $(\mathcal{V}_{\mp})^*$, this corresponds to an isomorphism 
$\mathcal{V}_-=(\mathcal{V}_+)^*$.\\ 
\indent 
Thus, under the decomposition $(V^*,V\times V^*)=(U_-,E_-)\times(U_+,E_+)$ into a product of a CR quaternionic and a co-CR quaternionic vector space, 
the canonical inner product corresponds to an isomorphism $E_-=E_+^*$ with respect to which $U_-$ is the annihilator of $U_+$\,.\\ 
\indent 
Therefore if $U'_+$ is a complement of $U_+$ in $E_+$ then its annihilator corresponds to a complement $U'_-$ of $U_-$ in $E_-$\,. 
Moreover, we may choose $U'_+$ so that $(U'_+,E_+)$ is a CR quaternionic vector space and, consequently, $(U'_-,E_-)$ is a co-CR quaternionic vector space. 
Then $U'_-\times U'_+$ is an isotropic complement to $V^*$ defining a linear $B$-field transformation which is as required. 
\end{proof}

\indent
Secondly, we consider the case when the sheaf of $(V^*,V\times V^*)$ is torsion.
 
\begin{prop} \label{prop:linear_gq_torsion} 
Let $V$ be a generalized quaternionic vector space such that the sheaf of $(V^*,V\times V^*)$ is torsion.\\
\indent
Then, up to a linear $B$-field transformation, $V$ is a product of generalized quaternionic vector spaces given by
complex symplectic vector spaces; moreover, the factors are unique, up to ordering. 
\end{prop}
\begin{proof}
Let $Z$ be the space of admissible linear complex structures on $V\times V^*$ and let $\mathcal{V}$ be the sheaf of $(V^*,V\times V^*)$\,.\\
\indent
Suppose, firstly, that the support of $\mathcal{V}$ is formed of two (necessarily, antipodal) points $\pm\J\in Z$;
equivalently, $V^*\cap \J(V^*)\neq0$ and for any $\mathcal{I}\in Z\setminus\{\pm\J\}$, we have $V^*\cap\mathcal{I}(V^*)=\{0\}$\,.\\
\indent
Let $U=V^*\cap \J(V^*)$ and let $\mathcal{K}\in Z$ be such that $\J\mathcal{K}=-\mathcal{K}\J$\,. Then $\mathcal{K}(V^*)$ is a $\J$-invariant isotropic complement of $V^*$, 
and $\mathcal{K}U=\mathcal{K}(V^*)\cap\J(\mathcal{K}(V^*))$\,.\\ 
\indent 
It follows that, up to a suitable linear $B$-field transformation, we have that $(V,\J)=(V_1,\J_1)\times(V_2,\J_2)$ 
such that $V_1\times V_1^*=U\oplus\mathcal{K}U$, and, also, $\J_1$ and $\J_2$ are given by a linear complex and symplectic structures, respectively. 
Consequently, $E=U\oplus\mathcal{K}U$ is a quaternionic vector subspace of $V\times V^*$ which is nondegenerate with respect to the canonical inner product, 
and its orthogonal complement is $V_2\times V_2^*$; in particular, the latter is preserved by $\mathcal{K}$\,. Therefore $\mathcal{V}$ contains the sheaf 
of $(V_2^*,V_2\times V_2^*)$ which, necessarily, is the sheaf of a co-CR quaternionic vector space. Thus, $V_2=\{0\}$\,; equivalently, $\J(V^*)=V^*$.\\ 
\indent 
Then, under the linear $B$-field transformation corresponding to $\mathcal{K}(V^*)$\,, we have that $\J$ and $\mathcal{K}$
are given by a linear complex structure $J$ and a linear symplectic structure $\o$\,, respectively, which,
as $\J$ and $\mathcal{K}$ anti-commute, define a linear complex symplectic structure on $V$.\\
\indent
In general, the support of $\mathcal{V}$ is $\{\pm\J_k\}_{k=1,\ldots,l}\subseteq Z$\,; denote 
$U_k=V^*\cap\J_k(V^*)$\,, $k=1,\ldots,l$\,. If $\mathcal{I}\in Z\setminus\{\pm\J_k\}_{k=1,\ldots,l}$ then 
from \cite[Theorem 3.1]{vq} it follows that we have $U_k\cap\mathcal{I}(U_k)=\{0\}$ and $U_k+\mathcal{I}(U_k)$
is a quaternionic vector subspace of $V\times V^*$, for any $k=1,\ldots,l$\,; moreover, the sum 
$\sum_{k=1}^l\bigl(U_k+\mathcal{I}(U_k)\bigr)$ is direct.\\
\indent
Now, inductively (and similarly to the case $l=1$), we obtain an orthogonal decomposition 
$V\times V^*=\bigoplus_{k=1}^l\bigl(U_k+\mathcal{I}(U_k)\bigr)$\,, and the proof follows. 
\end{proof}

\indent
We can, now, give the:

\begin{proof}[Proof of Theorem \ref{thm:linear_gq}\,] 
Let $Z$ be the space of admissible linear complex structures on $V\times V^*$ and let $\mathcal{V}$ be the sheaf of $(V^*,V\times V^*)$\,.\\
\indent
There exists a quaternionic vector subspace $E_t$ of $V\times V^*$ such that, if we denote $U_t=E_t\cap V^*$, then the sheaf of $(U_t,E_t)$ 
is the torsion part of $\mathcal{V}$\,; in particular, $\dim U_t=\tfrac12\dim E_t$\,. 
Consequently, $U_t^{\perp}=E_t^{\,\perp}+V^*$ and $\dim(E_t^{\,\perp}\cap V^*)=\tfrac12\dim E_t^{\,\perp}$.\\ 
\indent 
By using Proposition \ref{prop:from torsion_to_co-cr_q}\,, we obtain that $E_t$ is nondegenerate with respect to the canonical inner product.  
Hence, $E_t^{\,\perp}$ is a nondegenerate quaternionic vector subspace of $V\times V^*$ for which $E_t^{\,\perp}\cap V^*$ is a maximal
isotropic subspace.\\ 
\indent
We have thus shown that $(V,V\times V^*)=(E_t^{\perp}\cap V^*,E_t^{\perp})\times(U_t,E_t)$\,. Therefore, up to a 
linear $B$-field transformation, $V$ is a product of two generalized quaternionic vector spaces $V_{\!1}$ and $V_{\!2}$ such that
the sheaf of $(V_{\!1}^*,V_{\!1}\times V_{\!1}^*)$ is torsion free, whilst the sheaf of $(V_{\!2}^*,V_{\!2}\times V_{\!2}^*)$ is torsion.\\
\indent
By Propositions \ref{prop:linear_gq_bundle} and \ref{prop:linear_gq_torsion}\,, the proof is complete.
\end{proof}

\section{Generalized quaternionic manifolds} \label{section:gqm} 

\indent 
Recall \cite{Gua-thesis} (see, also, \cite{OrnPan-hologen}\,) that a \emph{generalized complex structure} on a manifold $M$ is an orthogonal complex structure on 
$TM\oplus T^*M$ for which the space of sections of its ${\rm i}$-eigenbundle is closed under the Courant bracket \cite{Courant}\,, defined by 
\begin{equation*}
[(X,\a)\,;(Y,\b)]=\bigl([X,Y]\,;\Lie_X\!\b-\Lie_Y\!\a-\tfrac12\dif(\iota_X\b-\iota_Y\a)\bigr)\;,
\end{equation*}
for any sections $(X,\a)$ and $(Y,\b)$ of $TM\oplus T^*M$.\\ 
\indent 
A \emph{generalized almost quaternionic structure} on a manifold $M$ is a linear quaternionic structure on $TM\oplus T^*M$ 
compatible to the canonical inner product.\\ 
\indent 
Let $M$ be a generalized almost quaternionic manifold and let $Z$ be the bundle of admissible generalized linear complex structures on $TM\oplus T^*M$.  
Note that, $Z$ is the sphere bundle of an oriented Riemannian vector bundle of rank three; in particular, its fibres are Riemann spheres.\\  
\indent 
Any connection $D$ on $TM\oplus T^*M$, compatible with its linear quaternionic structure and the inner product, induces a connection $\mathscr{K}$ on $Z$;  
in particular, $TZ=\mathscr{K}\oplus({\rm ker}\dif\!\p)$\,, where $\p:Z\to M$ is the projection.\\ 
\indent 
At each $J\in Z$, let $\J_J$ be the direct sum of $J$ and the linear complex structure of ${\rm ker}\,\dif\!\p_J$\,, where we have identified $\mathscr{K}_J=T_{\p(J)}M$, 
through $\dif\!\p$\,. Then $\J$ is a generalized almost complex structure on $Z$. 

\begin{defn} 
If $\J$ is a generalized complex structure then $(M,D)$ is a \emph{generalized quaternionic manifold} and $(Z,\J)$ is its \emph{twistor space}. 
\end{defn} 

\indent 
Obviously, the (classical) quaternionic manifolds (see \cite{IMOP} and the references therein) are generalized quaternionic. Also, from \cite{Bre-ghK} it follows that 
the generalized hyper-K\"ahler manifolds are generalized quaternionic.\\ 
\indent 
We add the following two classes of examples. 

\begin{exm} \label{exm:complex_symplectic} 
Let $(M,J,\o)$ be a complex symplectic manifold. Then, by using Example \ref{exm:gc}\,, we obtain a generalized almost quaternionic structure 
on $M$ whose bundle of admissible linear complex structure 
is $M\times S^2$.\\ 
\indent 
Furthermore, let $D$ be any compatible connection on $TM\oplus T^*M$ which induces the trivial connection on $M\times S^2$. Then a straightforward 
calculation shows that the induced generalized almost complex structure on $M\times S^2$ is integrable. Thus, $(M,D)$ is a generalized quaternionic manifold. 
\end{exm} 

\begin{exm} \label{exm:qKv} 
Let $(N^3,c,\nabla)$ be a three-dimensional Einstein--Weyl space and let $(M^4,g)$ be its heaven-space (see \cite{PB} and the references therein). 
Then $(M^4,g)$ is an anti-self-dual Einstein manifold, with nonzero scalar curvature, and $\nabla$ corresponds to a twistorial submersion 
$\phi:(M^4,g)\to(N^3,c,\nabla)$ \cite{Hit-complexmfds} (see \cite{fq-2}\,).\\ 
\indent 
As $\phi$ is horizontally conformal, at each $x\in M^4$, the differential of $\phi$ defines a linear co-CR quaternionic structure on $T_{\phi(x)}N$ 
and, hence, $\phi$ induces a generalized almost quaternionic structure on $M^4$. Recall (Example \ref{exm:cr_q}\,) that this is induced by 
the isomorphism $TM=\H\oplus\V^*$, where $\V={\rm ker}\dif\!\phi$\,, $\H=\V^{\perp}$, and we have used the musical isomorphisms 
determined by $g$. Furthermore, through this isomorphism, the Levi--Civita connection of $(M^4,g)$ corresponds to a connection $D$ on $TM\oplus T^*M$ 
with respect to which $(M^4,D)$ is a generalized quaternionic manifold.\\ 
\indent 
The twistor space of $(M^4,D)$ is obtained from the twistor space $(Z,J)$ of $(M^4,g)$\,, as follows.\\ 
\indent  
There exists a (pseudo-)K\"ahler metric $\widetilde{g}$ on $Z$ with respect to which the projection $\p:(Z,\widetilde{g})\to(M,g)$ is a Riemannian 
submersion with geodesic fibres (see \cite{Bes-Einstein} and the references therein).\\ 
\indent 
On the other hand, $\phi$ corresponds to a one-dimensional holomorphic foliation $\widetilde{\V}$ on $(Z,J)$\,, orthogonal to the fibres of $\p$\,.\\ 
\indent 
Then $\L(T^{0,1}Z+\widetilde{\V},{\rm i}\,\o)$\,, where $\o$ is the K\"ahler form of $(Z,J,\widetilde{g})$\,, is the $-{\rm i}$ eigenbundle of the 
generalized complex structure $\J$ on $Z$ such that $(Z,\J)$ is the twistor space of $(M^4,D)$\,. 
\end{exm}

\end{document}